\documentclass[10pt]{scrartcl}
\usepackage{amsmath, amsthm, amssymb}
\usepackage[T1]{fontenc}
\usepackage[utf8]{inputenc}
\usepackage{fullpage}

\usepackage{xcolor}

\newcommand{\eps}{\varepsilon}
\newcommand{\R}{\mathbb{R}}
\newcommand{\N}{\mathbb{N}}
\newcommand{\nat}{\in\N}
\newcommand{\subsub}{\subset\!\subset}
\newcommand{\Om}{\Omega}
\newcommand{\Ombar}{\overline{\Om}}
\newcommand{\Wnezom}{W_0^{1,2}(\Om)}

\newcommand{\Lzom}{L^2(\Om)}
\newcommand{\phii}{\varphi}
\newcommand{\thetaa}{\vartheta}
\newcommand{\ny}{\nu}
\newcommand{\my}{\mu}
\newcommand{\set}[1]{\left\{#1\right\}}
\newcommand{\io}{\int_\Om}
\newcommand{\ido}{\int_{\partial\Om}}
\newcommand{\intninf}{\int_0^\infty}
\newcommand{\intnT}{\int_0^T}
\newcommand{\na}{\nabla}
\newcommand{\amrand}{\rvert_{\partial\Om}}
\newcommand{\bdry}{\rvert_{\partial\Om}}
\newcommand{\dOm}{\partial\Om}
\newcommand{\pO}{\partial\Om}
\newcommand{\what}{\widehat{w}}
\newcommand{\nn}{\nonumber}
\newcommand{\ot}{\leftarrow}
\newcommand{\wto}{\rightharpoonup}
\newcommand{\wstarto}{\stackrel{\star}{\rightharpoonup}}

\newcommand{\liminfk}{\liminf_{k\to \infty}}
\newcommand{\liminfl}{\liminf_{l\to \infty}}
\newcommand{\limk}{\lim_{k\to\infty}}
\newcommand{\limtinf}{\lim_{t\to\infty}}

\newcommand{\ue}{u_\eps}
\newcommand{\une}{u_{0\eps}}
\newcommand{\uen}{\une}

\newcommand{\uek}{u_{\eps_k}}
\newcommand{\uet}{u_{\eps t}}
\newcommand{\norm}[2][ ]{\left\|#2\right\|_{#1}}
\newcommand{\sub}{\subset}
\newcommand{\subsubset}{\sub\sub}
\newcommand{\Laplace}{\Delta}
\newcommand{\Lap}{\Laplace}
\newcommand{\col}{\colon}
\newcommand{\bit}{\begin{itemize}}
\newcommand{\eit}{\end{itemize}}
\newcommand{\ddt}{\frac{d}{dt}}

\newcommand{\downto}{\searrow}

\newtheorem{theorem}{Theorem}
\newtheorem{remark}[theorem]{Remark}
\newtheorem{defn}[theorem]{Definition}
\newtheorem{lemma}[theorem]{Lemma}
\newtheorem{cor}[theorem]{Corollary}

\DeclareMathOperator{\esssup}{esssup}

\DeclareMathOperator{\esslim}{esslim}
\newcommand{\esslimtinf}{\esslim_{t\to \infty}}

\title{Equilibration of unit mass solutions to a degenerate parabolic equation with a nonlocal gradient nonlinearity}
\author{Johannes Lankeit\thanks{Institut f\"ur Mathematik, Universit\"at Paderborn, Warburger Str. 100, 33098 Paderborn, Germany; email: \mbox{johannes.lankeit@math.upb.de}}}
\begin{document}
 \maketitle 
 \begin{abstract}
\noindent
{\bf Abstract:} We prove convergence of positive solutions to 
\[
 u_t = u\Delta u + u\int_{\Omega} |\nabla u|^2, \qquad u\rvert_{\partial\Omega} =0, \qquad u(\cdot,0)=u_0
\]
 in a bounded domain $\Omega\subset \mathbb{R}^n$, $n\ge 1$, with smooth boundary in the case of $\int_\Omega u_0=1$ and identify the $W_0^{1,2}(\Omega)$-limit of $u(t)$ as $t\to \infty$ as the solution of the corresponding stationary problem. 
This behaviour is different from the cases of $\int_\Omega u_0<1$ and $\int_\Omega u_0>1$ which are known to result in convergence to zero or blow-up in finite time, respectively.\\
\noindent 
The proof is based on a monotonicity property of $\int_{\Omega} |\nabla u|^2$ along trajectories and the analysis of an associated constrained minimization problem.\\[0.3cm]
{\bf Math Subject Classification (2010):}  35B40, 35K55, 35K65
 \end{abstract}

\section{Introduction}
The problem 
\begin{equation}\label{eq:prob}
 u_t = u\Delta u + u\io |\na u|^2, \qquad u\amrand =0, \qquad u(\cdot,0)=u_0,
\end{equation}
which we want to investigate in this article with respect to the large-time behaviour of its solutions, combines two interesting and mathematically challenging mechanisms: degeneracy of diffusion and nonlocal contributions of gradient terms.

Already solutions to parabolic equations involving a sufficiently strong degeneracy in addition to a local source may display quite strange and unexpected large-time behaviour. For example, it is known that the problem 
\[
 u_t=u^p (\Lap u + u), \qquad u\amrand=0, \qquad u(\cdot,0)=0
\]
in bounded domains $\Om\sub\R^n$ with $\lambda(\Om)=1$ being the first Dirichlet eigenvalue of $-\Lap$ in $\Om$ 
admits oscillating solutions in the sense that $\limsup_{t\to\infty} \norm[L^\infty(\Om)]{u(\cdot,t)}=\infty$, $\liminf_{t\to\infty}\norm[L^\infty(\Om)]{u(\cdot,t)}=0$ if $p\ge 3$, see \cite{Wk_oscillating_degenerate}. This contrasts the expected stabilization behaviour which can be observed in related degenerate problems under a condition on the set of stationary solutions (see \cite{Wk_largetime_stability_degen,langlais_phillips}) or in the nondegenerate ($p=0$) case (see e.g. \cite{jendoubi,matano82,zelenjak}). 
At the same time, it differs from the alternative which likewise might naively be expected, that is convergence toward $\infty$ as $t\to \infty$ or blow-up in finite time, the latter most famously discovered in \cite{friedman_mcleod} for $\lambda(\Om)<1$ and $p=2$. 

Apart from degeneracies, another change in the equation that can alter the behaviour of solutions consists in the inclusion of a nonlocal term as source. For example, in the following equation involving gradient-dependent nonlocal terms, considered in \cite{Dlotko,Souplet}, nonlocality changes the overall appearance: In
\[
 u_t=\Lap u + u^m\bigg(\io |\na u|^2\bigg)^r
\]
under homogeneous Dirichlet or Neumann boundary conditions, with $r>0$, $m\geq 1$, pure gradient-blow-up cannot occur \cite[Thm 4.1]{Souplet}, whereas it is to be expected for sufficiently strong nonlinear contributions of the gradient in the corresponding local version of this equation \cite[Thm. 3.4]{Souplet}. 

Concerning further results on (non-degenerate) semilinear parabolic equations with nonlocal source terms, 
we refer to \cite{Souplet_nonloc_reac_diff}, \cite[Ch. V]{quittner_souplet} and references therein.

In degenerate equations with nonlocal sources, even 
in the case of nonlocal terms depending monotonically on $u$, 
the analysis may become rather involved, as 
can already be perceived from the rather strong assumptions that have been imposed on the initial data in order to prove existence
e.g. in \cite{LiXie} of solutions to
\[
 u_t = u^p \big( \Lap u + au\io u^q\big)
\]
for $p>1, q\geq 1$; they are such that monotonicity of the solution with respect to time follows. Under these assumptions, the authors then provide conditions for global existence or for blow-up of solutions, respectively. Typical results in the context of these equations are global existence, occurrence of blow-up and, occasionally, blow-up rates, see e.g. \cite{han_gao,deng_li_xie,deng_duan_xie,zhong}. 

The Dirichlet problem \eqref{eq:prob}, thus incorporating two effects which are quite delicate already on their own, 
has been investigated in \cite{klw1} in bounded domains $\Om\sub\R^n$ and it has turned out that solutions to \eqref{eq:prob} are highly sensitive to their initial mass $\io u_0$ being larger or smaller than $1$: They blow up globally after a finite time or decay to zero, respectively, as shown in \cite{klw1}. In the borderline case of unit inital mass, however, their behaviour must be different from both, \cite[Cor. 3.2]{klw1}.

Problem \eqref{eq:prob} arises in the context of evolutionary game dynamics (\cite{JohnMaynardSmith,dugatkin_reeve}), more precisely in the framework of replicator dynamics (\cite{TaJ78,ScS83}), applied to an infinite dimensional continuous setting. For more information on the modelling, we refer to the appendix of \cite{klw1} 
and to \cite{KPXY10,KPY08} and the references therein. 

In this setting, anyhow, $\Om$ is a set of ``strategies'' and, in the manner of a probability density, $u$ gives the relative frequency with which they are pursued in the population of ``players'' considered -- and accordingly the case of highest interest from an applications viewpoint is that of $\io u_0=1$.

Let us mention some results concerning explicit examples of solutions to \eqref{eq:prob} and solutions of a particular form: 
In \cite{KPY08} and \cite{PS09}, self-similar solutions to \eqref{eq:prob} were constructed in $\Om=\R$ and $\Om=\R^n$, respectively, and \cite{PV14} investigates self-similar solutions in $\Om=\R$ for a related model, where the Laplacian is perturbed by a time-dependent term involving first derivatives as well. 
In \cite{KPXY10,KP11}, stationary solutions of \eqref{eq:prob} were studied; however, it remained open if those accurately capture the long-term behaviour of solutions to \eqref{eq:prob}. 

A discussion of existence and qualitative properties of solutions to \eqref{eq:prob} for 
rather general initial data in bounded domains can be found in 
 the previous work \cite{klw1}, where the existence of solutions to \eqref{eq:prob} has been shown as limit of solutions to the actually non-degenerate parabolic problems
\begin{equation}\label{eq:epsprob}
 \uet = \ue\Delta \ue + \ue\min\set{\frac1\eps, \io |\na \ue|^2}, \qquad \ue\amrand =\eps, \qquad \ue(\cdot,0)=\une, 
\end{equation}
cf. Theorem \ref{thm:ex} below for a more precise statement.\\ 
As motivated above, the present article is devoted to a study of the case $\io u_0=1$. It has been shown in \cite{klw1} that corresponding solutions exist globally and that they satisfy $\io u(t) = 1$ for any $t\in[0,\infty)$. Hence neither of the kinds of qualitative behaviour witnessed for $\io u_0>1$ or $\io u_0<1$ is to be observed and one may hope for convergence to a nontrivial large-time limit. 
We will show that this actually is the case and, moreover, we can identify the limit as solution of the corresponding stationary problem:
\newpage
\begin{theorem}\label{thm:main} Let $\Om\sub\R^n$ be a bounded domain with smooth boundary. 
 Let $u_0\in\Wnezom$ satisfy (H1) - (H3) below and additionally $\io u_0=1$, 
 and let $u$ be a corresponding solution to \eqref{eq:prob} provided by Theorem \ref{thm:ex}. Let $\Phi$ denote the solution of $-\Laplace \Phi = 1$ in $\Om$, $\Phi\amrand=0$. Then 
\[
 \limtinf \norm[\Wnezom]{u(t) - \frac{\Phi}{\io \Phi}}=0.
\] 
\end{theorem}
\begin{remark}
 (H1)-(H3) are those assumptions concerning regularity, positivity and behaviour close to the boundary under which existence of a solution has been proven in \cite{klw1}.\\ 
 The solutions provided by Theorem \ref{thm:ex} are solutions that have been obtained as limit of solutions to \eqref{eq:epsprob} in the sense made precise in Theorem \ref{thm:ex}. 
 In stating the theorem for these solutions, we restrict ourselves to the class of solutions which have been shown to exist, rather than striving for the full generality of all locally positive weak solutions in the sense of Definition \ref{def:weaksoln}. 
 Whether there is a difference at all, depends on the question of uniqueness of solutions, which we do not want to pursue here.
\end{remark}

In Section \ref{sec:known} let us first review results from \cite{klw1} and give definitions, some of which we have already referred to. 
Along the way, we will observe that in our particular case of unit initial mass, the gradients of solutions $\ue$ of \eqref{eq:epsprob} have nonincreasing $\Lzom$-norms. This property can be used to slightly strengthen statements from \cite{klw1} and can be transferred to solutions $u$ of \eqref{eq:prob} provided by Theorem \ref{thm:ex}. We will undertake this transfer in Section \ref{sec:mon} to obtain one of the cornerstones for the proof of Theorem \ref{thm:main}.  
After that, in Section \ref{sec:min}, we turn our attention to a certain minimization problem
that is intimately 
connected with the identification of the large-time limit 
and will as well become important in deriving convergence at all. 
Finally, in Section \ref{sec:conv}, we combine the observations from Sections \ref{sec:mon} and \ref{sec:min} and give a proof of Theorem \ref{thm:main}. 
It is the fortunate combination of decrease of the Dirichlet integral with the conservation of total mass that makes it possible to capture the asymptotics in this case that lies so subtly between that of convergence to a trivial solution and that of mere non-global existence.

\section{Definitions and known results}\label{sec:known}

Let us begin by introducing a function that will appear throughout the article, and a norm that can be defined with its help.
\begin{defn}\label{def:phinorm}\label{def:Phi}
 Let $\Phi$ be the solution of 
\[
 \Delta\Phi=-1, \quad\mbox{in }\Om; \qquad \Phi\amrand=0.
\]

 For a measurable function $v\col \Om\to\R$ define
\[
 \norm[\Phi,\infty]{v} = \esssup_{\Om} | \frac v \Phi|.
\]
\end{defn}

Controlling this norm of a function means providing estimates for its value at a point $x\in\Om$ in terms of the distance $d(x)=dist(x,\pO)$ of $x$ to the boundary of the domain, because there exist $c_0,c_1>0$ such that $c_0 d\le \Phi\le c_1 d$. In the construction of solutions as performed in \cite{klw1}, positivity of initial data and their behaviour close to the boundary play a role; more precisely we require the following:

 \bit
\item[(H1)]
  $u_0 \in L^\infty(\Omega) \cap W_0^{1,2}(\Omega)$ and
\item[(H2)]
  $u_0 \ge 0$ and $\frac{1}{u_0} \in L^\infty_{loc}(\Omega)$ and
\item[(H3)]
 $\norm[\Phi,\infty]{u_0}<\infty$.
\eit

Throughout the article, let us fix $u_0$ with the properties (H1)-(H3) and $\io u_0=1$.

\begin{lemma}\label{lem:approxinit}
 Let $L>\max\set{\io|\na u_0|^2, \norm[\Phi,\infty]{u_0}}$. Then there exist $\eps_0>0$ and a family $(\une)_{\eps\in(0,\eps_0)}\subset C^3(\Ombar)$ with the following properties:
\begin{align} &\uen \geq \eps\label{geqeps}, \qquad \uen\bdry=\eps,\\
 &\mbox{for any compact set } K\subset \Om \mbox{ there is }C_K>0 \mbox{ such that } \liminf_{\eps\searrow 0} \inf_K \uen \geq C_K,\label{Kptbelow}\\
 &\uen\to \ue \quad \mbox{in } W^{1,2}(\Om) \qquad \mbox{as }\eps\searrow 0,\label{wez}\\
 &\limsup_{\eps\searrow 0} \norm[\Phi,\infty]{\uen-\eps}\leq L,\qquad \Laplace \uen\bdry=-\io |\na \uen|^2,\qquad \io \uen=\io u_0\label{phinorm_lapbdry_initint}.
\end{align}
\end{lemma}
\begin{proof}
 A proof of this approximation property is sketched in \cite[Lemma 2.4]{klw1}. 
\end{proof}
Corresponding to initial data $u_0$ given above, in the following let us fix $\eps_0>0$ and a family $\set{\une}_{\eps\in(0,\eps_0)}$ with the properties described in Lemma \ref{lem:approxinit}.

\begin{lemma}\label{lem:exapprox}
 For sufficiently small $\eps\in(0,1)$, the problem \eqref{eq:epsprob} with initial data as just described has a unique classical solution in $\Om\times(0,\infty)$.
\end{lemma}
\begin{proof}
 \cite[Lemma 2.5]{klw1}.
\end{proof}
From now on, we will alway use $\ue$ to denote this solution.

Estimating it from below is rather straightforward.
\begin{lemma}\label{lem:greatereps}
 The solution $\ue$ from Lemma \ref{lem:exapprox} satisfies $\ue\ge \eps$ on $\Om\times(0,\infty)$.
\end{lemma}
\begin{proof}
 With the aid of \eqref{geqeps} and the boundary condition $\ue\amrand =\eps$, this follows by comparison.
\end{proof}

Given bounds on the spatial gradient, we can estimate $\ue$ also from above; this time the bound does not depend on $\eps$. 
\begin{lemma}\label{lem13}
  For all $M>0$ and $C_0>0$ there exists $C_1(M,C_0)>0$ with the following property: 
  If  
  \[
	u_{0\eps} \le M \quad \mbox{in } \Omega \qquad \mbox{and} \qquad \int_0^T \io |\nabla \ue|^2 \le C_0
  \]
  holds for some $\eps\in (0,\eps_0)$ and $T \in (0,\infty]$ then we have
  \(
	\ue \le C_1(M,C_0)$  in $\Omega \times (0,T).
  \)
\end{lemma}
\begin{proof}
 This is Lemma 2.6 of \cite{klw1}. It has been obtained by a comparison argument there.
\end{proof}

\begin{lemma}\label{lem:massdecrease}
For any $t>0$ we have $\ddt \io \ue(t)\leq 0$ and $\io \ue(t) \leq 1$.
\end{lemma}
\begin{proof}
 Using \eqref{eq:epsprob} and integrating by parts we obtain 
\begin{align}\label{eq:ddtioue}
 \ddt \io \ue =& \io \ue\Delta \ue + \io \ue \min\set{\frac1\eps, \io |\na \ue|^2}\nn\\
 =& -\io |\na\ue|^2 + \ido \ue \na\ue\cdot \ny + \io \ue \min\set{\frac1\eps, \io |\na \ue|^2}\nn\\
\le& -\io |\na\ue|^2 +\io \ue \io |\na\ue|^2 \nn\\
=& \left(\io \ue-1\right) \io |\na\ue|^2 \qquad \mbox{ for } t>0,
\end{align}
where $\ny$ denotes the outer unit normal and where we used that $\na\ue \cdot \ny\le 0$ on $\dOm$ because of Lemma \ref{lem:greatereps}. Since we have $\io \ue(0)  \le 1$, an ODE comparison argument shows that $\io\ue(t)\le 1$ for all $t>0$ and hence by \eqref{eq:ddtioue} also $\ddt \io\ue(t)\leq 0$ for all $t>0$.
\end{proof}

\begin{lemma}\label{lem:epsmon}
The mapping defined by
\[
 [0,\infty)\ni t\mapsto\io |\na\ue(t)|^2
\]
is monotonically nonincreasing. In particular, for any $t>0$
\[
 \io |\na\ue(t)|^2 \le \io |\na\uen|^2.
\]
\end{lemma}
\begin{proof}
 We multiply equation \eqref{eq:epsprob} by $\frac{\uet}{\ue}$ and integrate by parts. On the lateral boundary $\uet=0$ and by Lemma \ref{lem:massdecrease}, $\io\uet$ is nonpositive. Hence
 \begin{equation}\label{eq:uet2ubd}
  \io \frac{\uet^2}{\ue} = \io \uet \Delta \ue + \io \uet \min\set{\frac1\eps \io |\na\ue|^2} 
 \leq  - \frac12 \ddt \io |\na \ue |^2 
 \end{equation}
on $(0,\infty)$. Since obviously $\io \frac{\uet^2}{\ue}$ is nonnegative, this proves the assertion.
\end{proof}

Having gathered some knowledge about the approximate solutions, let us now turn to the limiting case and make precise what we mean by a solution:

\begin{defn}\label{def:weaksoln}
  By a {\em weak solution} of \eqref{eq:prob} in $\Omega \times (0,\infty)$ we mean a nonnegative function
  \begin{align}\label{eq:solndefregularity}
	u\in L^\infty_{loc}(\bar\Omega \times [0,\infty)) \, \cap \, L^2_{loc}([0,\infty); W_0^{1,2}(\Omega)) 
	\qquad \mbox{with} \qquad u_t \in L^2_{loc}(\bar\Omega \times [0,\infty))
  \end{align}
  which satisfies
  \begin{equation}\label{0vw}
	-\int_0^T \io u\varphi_t + \int_0^\infty \io \nabla u \cdot \nabla (u\varphi)
	=\io u_0 \varphi(\cdot,0) + \int_0^\infty \Big( \io u\varphi \Big) \cdot \Big( \io |\nabla u|^2 \Big)
  \end{equation}
  for all $\varphi \in C_0^\infty(\Omega \times [0,\infty))$.\\
  A weak solution $u$ of \eqref{eq:prob} in $\Omega \times (0,\infty)$ will be called {\em locally positive} if 
  $\frac{1}{u} \in L^\infty_{loc}(\Omega \times (0,\infty))$.
\end{defn}
\begin{remark}\label{rem:lsgdef}
i)  Since $u \in L^2_{loc}([0,\infty);W_0^{1,2}(\Omega))$ and $u_t \in L^2_{loc}(\bar\Omega \times [0,\infty))$
  imply that $u \in C^0([0,\infty);L^2(\Omega))$, (\ref{0vw}) is equivalent
  to requiring that $u(\cdot,0)=u_0$, and that
\begin{equation}\label{0w}
	\int_0^\infty \io u_t \varphi + \int_0^\infty \io \nabla u \cdot \nabla (u\varphi)
	=\int_0^\infty \Big( \io u\varphi \Big) \cdot \Big( \io |\nabla u|^2 \Big)
\end{equation}
  holds for any $\varphi \in C_0^\infty(\Omega \times (0,\infty))$.\\
ii) Furthermore, close inspection of \eqref{0w} and density arguments show that we actually may use any function $\phii\in L^2((0,\infty),W^{1,2}_0(\Om)) \cap L^\infty(\Om\times(0,\infty))$ as test function.
\end{remark}

\begin{theorem}
\label{thm:ex}
 Let $u_0$ be as fixed above. (In particular that means that $u_0$ satisfies (H1), (H2) and (H3) as well as $\io u_0=1$.) Then there exists a sequence $(\eps_k)_{k\in\N}\to 0$ and a locally positive weak solution $u$ of \eqref{eq:prob} in $\Om\times(0,\infty)$ such that 
\begin{align}
 &&\ue&\to u &\quad&\mbox{in } C_{loc}^0([0,\infty);\Lzom)&\quad&\mbox{ and a.e. in } \Om\times(0,\infty)\label{eq:convucl2ae}\\
 &&\uet&\wto u_t &\quad&\mbox{in } L^2_{loc}(\Om\times[0,\infty))\label{eq:convut}\\
 &&\na\ue&\wstarto \na u&\quad&\mbox{in } L^\infty([0,\infty);\Lzom)\label{eq:convnaulilz}\\
 &&\na\ue&\to \na u&\quad&\mbox{in } L^2_{loc}(\Om\times [0,\infty)) &\quad&\mbox{ and a.e. in } \Om\times(0,\infty)\label{eq:convnau}\\
 &&\io  |\na\ue(x,\cdot)|^2 dx &\to \io |\na u(x,\cdot)|^2dx&\quad& \mbox{in } L^1_{loc}([0,\infty))&\quad&\mbox{ and a.e. in }[0,\infty)\label{eq:convintnau}
\end{align}
as $\eps=\eps_k\to 0$.
\end{theorem}
\begin{proof}
Due to the above information that we have obtained for solutions with unit initial mass, we can begin the proof without referring to \cite[Lemma 2.7]{klw1}. The rest, however, follows the proof of \cite[Thm. 2.11]{klw1} very closely, where details can be found. Here we may therefore present the line of reasoning more briefly:\\
Let $T>0$. Lemma \ref{lem:epsmon} asserts boundedness of $\io |\na\ue|^2$ with a bound that is independent of $\eps$ due to \eqref{wez}. Lemma \ref{lem13}
converts this into uniform boundedness of the solutions $\ue$ of \eqref{eq:epsprob} on $[0,T)$. Upon integration of \eqref{eq:uet2ubd} we obtain a bound on $\int_0^T \io \frac{\uet^2}\ue$ and hence on $\intnT\io \uet^2$. Together with Lemma \ref{lem:epsmon} we can use this to infer boundedness in $C^{\frac12}([0,T);L^2(\Om))$. From the mentioned bounds we can deduce the existence of a subsequence $(\uek)_{k\nat}$ of $(\ue)_{\eps\in(0,1)}$ converging in $C^0([0,T);\Lzom)$, such that $\uet\wto u_t$ in $L^2(\Om\times(0,T))$ and $\na\ue\wto \na u$ in $L^2(\Om\times(0,T))$ as $k\to\infty$. 
By a diagonal sequence argument we hence finally obtain $u\col \Om\times [0,\infty)\to\R$ 
enjoying the regularity properties required for a solution in Definition \ref{def:weaksoln} and arising as limit of the $\ue$ in the sense of \eqref{eq:convucl2ae}, \eqref{eq:convut} and 
\begin{equation}\label{eq:convnauweak}
  \na\ue\wto\na u \quad\mbox{in } L^2_{loc}(\Ombar\times[0,\infty)) \qquad\mbox{as }\eps=\eps_k\downto 0.
\end{equation}
Combining the bound from Lemma \ref{lem:epsmon} with \eqref{eq:convnauweak} for identification of the limit, we easily arrive at \eqref{eq:convnaulilz}.\\ 
Comparison from below with $(x,t)\mapsto \phi(x) \frac{c(\Om')}{1+c(\Om')t}$, where $\Om'\subsub \Om$ is a smoothly bounded subdomain of $\Om$, $\phi$ solves $-\Laplace \phi=1$ in $\Om'$ under homogeneous Dirichlet boundary conditions and $c(\Om')$ denotes the common lower bound for $\uen$ provided by \eqref{Kptbelow}, also ensures the local positivity of $u$.\\
To derive \eqref{eq:convnau}, we then choose $K\subsubset\Om$ and $T'\in(0,T)$ and harness a nonnegative function $\psi\in C_0^\infty(\Om)$ which satisfies $\psi\equiv 1$ on $K$, so that 
\begin{align*}
 \int_0^{T'}\int_K |\na \ue-\na u|^2\leq& \int_0^{T'}\io |\na \ue-\na u|^2\psi \\=&-\int_0^{T'}\io (\ue-u)\Delta \ue \psi -\int_0^{T'}\io (\ue-u)\na\ue\cdot\na\psi-\int_0^{T'}\io \na(\ue-u)\cdot\na u\psi,
\end{align*}
where we can handle the second and third integral by \eqref{eq:convucl2ae} and \eqref{eq:convnaulilz}. In $\int_0^{T'}\io (\ue-u)\Delta \ue \psi$, we employ \eqref{eq:epsprob}, \eqref{eq:convucl2ae} and Lemma \ref{lem:epsmon} and estimate 
\[
 \left\lvert\int_0^{T'} \io (\ue-u) \cdot \frac{u_{\eps t}}{\ue} \cdot \psi\right\rvert \le \left( \int_0^{T'} \io \frac{u_{\eps t}^2}{\ue} \right)^\frac{1}{2} \cdot
	\left( \int_0^{T'} \io \frac{(\ue-u)^2}{\ue} \cdot \psi^2 \right)^\frac{1}{2}\le C\norm[\infty]{\psi} \left(\int_0^{T'} \io (\ue-u)^2\right)^{\frac12},
\]
again with the help of \eqref{eq:uet2ubd} and the local lower estimate. 
Another application of \eqref{eq:convucl2ae} then is sufficient to finally derive \eqref{eq:convnau}.\\
As a result of testing the equation by the singular weight $\ue^{q-1}$ for some $q\in(0,1)$, we obtain 
\[
 \frac1q\frac{d}{dt} \io \ue^q \leq -q \io \ue^{q-1} |\na \ue|^2 +\io \ue^q\io |\na \ue|^2
\]
and hence $\intnT \io \ue^{q-1} |\na\ue|^2 \leq C(T)$, so that we gain control over the integral of $|\na\ue|^2$ over sets where $\ue$ is small in the sense below. Namely, comparison with $A\Phi+\eps$ for some $A>0$ (whose possibility can be traced back to \eqref{phinorm_lapbdry_initint} and thus (H3)) asserts that the vicinity of $\dOm$ is such a set, so that 
\[\intnT\int_{\Om\setminus K} |\na \ue|^2 \leq \sup_{\Om\setminus K} \ue^{1-q} \intnT\int_{\Om\setminus K} |\na \ue|^2 \ue^{q-1}\leq C(T) \sup_{\Om\setminus K} \ue^{1-q} \] 
becomes small uniformly in $\eps>0$, provided the compact set $K\sub \Om$ is chosen sufficiently large (this is Lemma 2.8 of \cite{klw1}).  Together with \eqref{eq:convnau} this makes it possible to deduce also \eqref{eq:convintnau}. \\
Finally, a combination of the convergence results with equation \eqref{eq:epsprob} for $\ue$ shows that $u$ is a weak solution to \eqref{eq:prob} in the sense of Definition \ref{def:weaksoln}.
\end{proof}

In \cite{klw1}, the corresponding theorem states existence on a possibly small time interval $(0,T)$ and is accompanied by a continuation argument.  
Since we had obtained a bound on $\io |\na \ue|^2$ before, it was possible to have Theorem \ref{thm:ex} state existence of a solution $u$ which is approximated by a sequence $(u_{\eps_j})_{j\in\N}$ of solutions to \eqref{eq:epsprob} on the whole set $\Om\times (0,\infty)$. Of course, this renders any further continuation argument unnecessary for the present case.

The following lemma, though already observed in \cite[Thm 3.1, Cor 3.2]{klw1}, is of central importance to the current article, so that we give at least an outline of the proof also here:
\begin{lemma}\label{lem:massconservation}
Any 
 weak solution $u$ of \eqref{eq:prob} in the sense of Definition \ref{def:weaksoln} satisfies 
 \[
  \io u(t)=1 
 \]
 for all $t>0$.
\end{lemma}
\begin{proof}
By Remark \ref{rem:lsgdef} (ii), we may insert the function defined by 
\[
 \chi(x,\tau)=\begin{cases}
0,& t<s-\delta\\
1+\frac{\tau-s}\delta,&s-\delta\leq\tau<s\\
1, & s\leq\tau<t \\
             1-\frac{\tau-t}\delta,&t\leq \tau<t+\delta\\
	      0,&\tau\geq t+\delta
\end{cases}
\]
for $0<s<t<T$ and $0<\delta<\min\set{s,T-t}$ as test function in \eqref{0w}. 
Since $u\in C^0([0,\infty),\Lzom)$ according to Remark \ref{rem:lsgdef} (i), letting $\delta\to0$ we obtain that $y(t) = \io u(t)$ defines an absolutely continuous function on $[0,\infty)$ and 
\begin{equation}\label{eq:yeq}
 y(t)-y(s) = \int_s^t \left((y(\tau)-1)\io |\na u(x,\tau)|^2 dx \right) d\tau
\end{equation}
is satisfied for all $s,t\ge 0$. Inserting $s=0$, applications of Gronwall's inequality to $y-1$ and $1-y$, respectively, conclude the proof.
\end{proof}

\section{Monotonicity of $t\mapsto \io |\na u(t)|^2$}\label{sec:mon}
The monotonicity of $t\mapsto \io |\na\ue(t)|^2$ from Lemma \ref{lem:epsmon} can be carried over to $u$ in the following sense:

\begin{lemma}\label{lem:mon}
 Let $u$ be a solution provided by Theorem \ref{thm:ex}. There exists a set $N\sub(0,\infty)$ of measure $0$ such that for every $t_1\in[0,\infty)\setminus N$ and each $t_2>t_1$
 \[
  \io |\na u(t_2)|^2 \le \io |\na u(t_1)|^2.
 \]
\end{lemma}
\begin{proof}
 Let $\ue$ be solutions to \eqref{eq:epsprob} approximating $u$ as in Theorem \ref{thm:ex} and let $N,N_2 \sub (0,\infty)$ be sets of measure zero such that  
\begin{equation}\label{eq:convergenceoutsidemeasurezero}
 \io |\na \ue(t)|^2\to \io |\na u(t)|^2 \quad \mbox{for any }t\in [0,\infty)\setminus N 
\end{equation}
as $\eps\to 0$ and that, with $C=\esssup_{t>0} \io |\na u(t)|^2<\infty$,
\[
 \io |\na u(t)|^2 \le C \quad \mbox{for any }t\in [0,\infty)\setminus N_2
\]
hold. These sets are provided as part of Theorem \ref{thm:ex} by \eqref{eq:convintnau} and \eqref{eq:convnaulilz}, respectively.
Let $t_1\in [0,\infty)\setminus N$ and let $t_2>t_1$. Let $(\tau_k)_{k\in\N}\subset [t_1,\infty)\setminus (N\cup N_2)$ be a sequence with limit $t_2$. 
As $t_1, \tau_k\notin N$, we can infer from Lemma \ref{lem:epsmon} and \eqref{eq:convergenceoutsidemeasurezero} that
\begin{equation}\label{eq:intnaualmostmonotone}
  \io |\na u(\tau_k)|^2 \ot \io |\na \ue(\tau_k)|^2 \le \io |\na \ue(t_1)|^2 \to  \io |\na u(t_1)|^2
\end{equation}
as $\eps=(\eps_{l})_l \to 0$.\\
Since $\tau_k\notin N_2$, the sequence $(\na u(\tau_k))_{k\in\N}$ is bounded in $\Lzom$ and a subsequence converges weakly in $\Lzom$. Because $u\in C^0_{loc}([0,\infty),\Lzom)$ by \eqref{eq:convucl2ae}, we can identify the limit and hence have $\na u(\tau_k)\wto \na u(t_2)$. Making use of weak lower semicontinuity and \eqref{eq:intnaualmostmonotone}, we obtain the claim from 
\[
 \io |\na u(t_2)|^2 \le \liminf_{k\to \infty} \io |\na u(\tau_k)|^2 \le \io |\na u(t_1)|^2.\qedhere
\]
\end{proof}

\begin{cor}\label{cor:naubd}
 The set $\set{\io |\na u(t)|^2: t>0}$ is bounded.
\end{cor}

\section{A minimization problem}\label{sec:min}
In order to obtain convergence statements, it will be important to estimate $\io |\na u(t)|^2$ and to identify the possible limit of $u(t)$. 
Both of these aims will be feasible by the following minimization result.
  
\begin{theorem}\label{thm:min}
 Consider the set 
\[
 M=\set{v\in\Wnezom; \io v=1}
\]
and the functional
\[
 J(v) = \io |\na v|^2, \qquad v\in\Wnezom.
\]
Then the minimization problem 
\[
 \min_{v\in M} J(v)
\]
has a unique minimizer $w\in M$. With $\Phi$ as in Definition \ref{def:Phi}, 
this minimizer satisfies 
\begin{equation}\label{eq:thisisminimizer}
 w= \frac{1}{\io \Phi} \Phi,
\end{equation}
 and we have 
\[  \min_{v\in M} J(v) = \frac{1}{\io\Phi}.\]
\end{theorem}
\begin{proof} Existence of a unique minimizer follows from coercivity and strict convexity of $J$ in combination with convexity of $M$ by straightforward arguments. In order to make the article self-contained, we give the short proofs below. Afterwards we employ a variational argument to show \eqref{eq:thisisminimizer}.\\
 {\bf Uniqueness:} Let $w, \what \in M$ be two minimizers. Let $\thetaa\in(0,1)$. Then $\thetaa w + (1-\thetaa)\what\in M$, and by the Cauchy-Schwarz inequality we have
\begin{align}\label{eq:Jconvex}
 J(\thetaa w+(1-\thetaa) \what) =& \thetaa^2\io|\na w|^2 + 2\thetaa(1-\thetaa)\io \na w\cdot\na \what + (1-\thetaa)^2\io|\na\what|^2\nn\\
\le&\thetaa^2\io|\na w|^2 + 2\thetaa(1-\thetaa) \left(\io|\na w|^2\right)^{\frac12}\left(\io|\na \what|^2\right)^{\frac12}+(1-\thetaa)^2 \io|\na \what|^2\nn\\
=& \left(\thetaa \left(\io|\na w|^2\right)^{\frac12}+(1-\thetaa)\left(\io|\na \what|^2\right)^{\frac12}\right)^2 \nn\\
=& \left(\thetaa (\min_{v\in M} J(v))^{\frac12}+(1-\thetaa)(\min_{v\in M} J(v))^\frac12\right)^2 = \min_{v\in M} J(v).\nn
\end{align}
Since having a strict inequality in this formula would contradict the definition of a minimizer, actually equality holds and the Cauchy-Schwarz inequality allows us to conclude that $\what=\lambda w$ for some $\lambda\in \R$.
Since $w, \what\in M$, we have $1=\io \what = \lambda \io w =\lambda\cdot 1$, that is, $w=\what$ and uniqueness of the minimizer is proven.\\
 {\bf Existence:} Let $(w_k)_{k\in\N}\subset M\subset \Wnezom$ be a sequence such that $J(w_k)\to \inf_{v\in M} J(v)$. Then $(J(w_k))_k$ and hence $(w_k)_k$ is bounded. Therefore there is a subsequence $(w_{k_l})_l$ converging weakly in $\Wnezom$ to some $w\in \Wnezom$.\\
 The set $M$ is convex and closed, hence weakly sequentially closed, thus $w\in M$.\\
 As the functional $J$ is continuous and convex
, it is weakly sequentially lower semicontinuous and 
\[
 \inf_{v\in M} J(v)\leq J(w)\le \liminfl J(w_{k_l}) = \inf_{v\in M} J(v).
\]
 Thus $w=\min_{v\in M} J(v)$. \\
 {\bf Properties:} Let $\phii\in W_0^{1,2}(\Om)$ and $\lambda_0:=\left\lvert\io \phii\right\rvert^{-1}$. If we define the polynomial functions $P(\lambda):=\io|\na w|^2 + 2\lambda \io \na w\cdot \na \phii + \lambda^2\io |\na \phii|^2$ and $Q(\lambda):=1+2\lambda\io\phii+\lambda^2\io \phii^2$, then the function 
\[
 (-\lambda_0,\lambda_0)\ni \lambda \mapsto \frac{P(\lambda)}{Q(\lambda)}=\io \left\lvert\nabla \frac{w+\lambda\phii}{1+\lambda\io\phii} \right\rvert^2
\]
is differentiable and, since $\frac{w+\lambda\phii}{1+\lambda\io\phii}=\frac{w+\lambda\phii}{\io(w+\lambda\phii)}\in M$, has a minimum at $\lambda=0$, so that $P'(0)Q(0)-P(0)Q'(0)=0$, i.e.
\[
 2\io \na w\cdot \na\phii - 2\io \phii\io|\na w|^2=0.
\]
If we let $\my=\io |\na w|^2$, then due to $\phii\in W_0^{1,2}(\Om)$ being arbitrary, $\psi:=\frac{w}{\my}$ evidently is a weak solution of $-\Delta \psi =1$, $\psi\amrand=0$ and has to coincide with $\Phi$ from Definition \ref{def:Phi}. Moreover, $\io\Phi = \io \frac{w}{\my} = \frac1{\my}$ implies $\my=\frac1{\io\Phi}$ as well as $w = \frac{\Phi}{\io\Phi}$.
\end{proof}


\section{Convergence. Proof of Theorem \ref{thm:main}}\label{sec:conv}
\subsection{Convergence of $\io |\na u(t)|^2$}

\begin{lemma}\label{lem:limitofnorm} For the solution $u$ provided by Theorem \ref{thm:ex}, the limit 
 \[
  A := \limtinf \io |\na u(t)|^2
 \]
 exists. 
Furthermore, $A$ coincides with the minimum computed in Section \ref{sec:min}, that is, $A=\frac1{\io \Phi}$.
\end{lemma}
\begin{proof}
Let $(s_k)_{k\in\N}\subset (0,\infty)\setminus N$, where $N$ is the null set from Lemma \ref{lem:mon}, be a sequence with $s_k\to\infty$ as $k\to\infty$. Then $\int|\na u(s_k)|^2$ is monotone decreasing and hence converges to its (nonnegative) infimum. \\ 
Let $A:= \esslimtinf \io |\na u(t)|^2$. For almost all $t>0$, $\io |\na u(t)|^2\ge A$. We will first determine the value of $A$ and finally show that actually $A$ is the limit.

For all $t>0$, we have $\io u(t)=1$ and as a consequence of Section \ref{sec:min} 
\begin{equation}\label{eq:nauge}
\io |\na u(t)|^2\ge \frac1{\io\Phi},
\end{equation}
 and hence also $A\ge \frac1{\io\Phi}$.

Assume $A>\frac1{\io\Phi}$, that is $\io \Phi > \frac1A$. Then we can find a subdomain $\Om'\subsub \Om$ such that the solution $\phi$ of $-\Laplace \phi = 1$ in $\Om'$, $\phi|_{\partial \Om'}=0$ still satisfies 
\begin{equation}\label{eq:choosephi}
\io \phi > \frac1A
\end{equation}
where we understand $\phi$ as being zero on $\Om\setminus \Om'$.
With this extended definition at hand, for $T>0$ and with $\chi_{[0,T]}$ denoting the characteristic function of the interval $[0,T]$, we take 
\[
 \phii = \frac{\phi}{u}\chi_{[0,T]} \in L^2((0,\infty);\Wnezom)\cap L^\infty(\Om\times(0,\infty))
\]
as test function, which is possible by Remark \ref{rem:lsgdef} ii). Inserting this function into 
\[
 \intninf \io u_t \phii +\intninf\io \na u\cdot\na (u\phii) = \intninf \io u\phii\io |\na u|^2,
\]
we obtain 
\begin{equation}\label{eq:getestet}
 \io \phi\intnT \frac{u_t}{u} + \intnT\io \na u\cdot\na\phi = \intnT\io \phi\io |\na u|^2,
\end{equation}
where evaluation of $\intnT \frac{u_t}u$ 
and integration by parts transform the left hand side according to
\begin{align}\label{eq:estimatephilnu}
 \io \phi\intnT \frac{u_t}{u} + \intnT\io \na u\cdot\na\phi =& \io \phi\left[\ln u(T)- \ln u_0\right] + \intnT\int_{\partial \Om'} u\underbrace{\na\phi\cdot\ny}_{\le 0} - \intnT \int_{\Om'} u\Laplace \phi\nn\\
 \le& \io \phi \ln u(T) - \io \phi \ln u_0 + \intnT \int_{\Om'} u\nn\\
\le& \io \phi \ln u(T) - \io \phi \ln u_0 + T,
\end{align}
because $\int_{\Om'} u(t) \le 1$ for all $t>0$.\\
Combining this again with \eqref{eq:getestet} and the fact that $\io |\na u(t)|^2\ge A$ for almost all $t>0$, we see that 
\[
 \io \phi \ln u(T) \ge \io \phi\ln u_0 + \intnT \left(\io \phi\io |\na u|^2 - 1\right) \ge \io \phi\ln u_0 + \intnT \left(A\io \phi - 1\right),
\]
where the last integrand is positive by \eqref{eq:choosephi}. By $\phi$ being compactly supported and (H2), $\io \phi\ln u_0$ is finite, 
and we may conclude that $\io \phi \ln u(T) \to \infty$ as $T\to \infty$.
But this is a contradiction to 
\[
 \infty > \norm[\infty]{\phi} = \norm[\infty]{\phi}\io u(T) \ge \io \phi \ln u(T),
\]
and hence $A \le \frac1{\io \Phi}$.\\
We still have to show that $A$ is the limit. For this, 
let $(\tau_k)_{k\in\N}\sub (0,\infty)$ be a sequence with $\limk \tau_k=\infty$, where $\tau_k \in N$ is not excluded. Let $(t_k)_{k\in\N}\sub (0,\infty)\setminus N$ be such that $t_k<\tau_k$ for each $k\in\N$ and $\limk t_k=\infty$. Because \eqref{eq:nauge} holds for any $t>0$ and Lemma \ref{lem:mon} is applicable to $t_k\notin N$, $\tau_k>t_k$, we obtain
\[
  A =\frac{1}{\io \Phi} \le \io |\na u(\tau_k)|^2\le \io |\na u(t_k)|^2\to A
\]
as $k\to \infty$.
\end{proof}

\subsection{Convergence of $u$}
\begin{lemma}\label{lem:weakconv}
Let $u$ be a solution provided by Theorem \ref{thm:ex}. Then 
 \begin{equation}\label{eq:weakconv}
  u(t) \wto \frac{1}{\io \Phi} \Phi
 \end{equation}
 weakly in $\Wnezom$ as $t\to\infty$.
\end{lemma}
\begin{proof}
 Assume \eqref{eq:weakconv} to be false. Then there exists a sequence $t_k\to \infty$ such that no subsequence $(t_{k_l})_{l\in\N}$ converges to $\frac{1}{\io \Phi} \Phi$ weakly in $\Wnezom$. 
 But Corollary \ref{cor:naubd} ensures the existence of a weakly convergent subsequence $(u(t_{k_l}))_{l\in\N}$ of $(t_k)_{k\in\N}$, whose limit we call $u_\infty$. 
 From this weak convergence, we also obtain $\io u(t_{k_l}) \to \io u_\infty = 1$.
 By the minimum results from Section \ref{sec:min}, weak lower semicontinuity of the norm and the convergence result from Lemma \ref{lem:limitofnorm} 
 \[
   \frac1{\io\Phi} \le \io |\na u_\infty|^2 \le \liminfk \io |\na u(t_k)|^2 = \frac1{\io\Phi}, 
 \]
 hence $u_\infty$ must coincide with the unique minimizer from Theorem \ref{thm:min}, contradicting the choice of $(t_k)_k$. 
\end{proof}

\subsection{Proof of Theorem \ref{thm:main}}
\begin{proof}[Proof of Theorem \ref{thm:main}]
 In Lemma \ref{lem:weakconv} and Lemma \ref{lem:limitofnorm}, respectively, we have obtained weak convergence of $u(t)$ as $t\to\infty$ in $\Wnezom$ with limit $\frac{1}{\io \Phi}\Phi$  and convergence of the norm $\norm[\Wnezom]{u(t)}=\io |\na u(t)|^2$ to the norm $\norm[\Wnezom]{\frac{1}{\io \Phi}\Phi}=\frac{1}{\io \Phi}$ of the limit. Together, these imply convergence 
\[
 u(t)\to \frac{\Phi}{\io\Phi}\qquad \mbox{as }t\to \infty 
\]
in the Hilbert space $\Wnezom$.
\end{proof}

\section{Acknowledgements}
The author would like to thank M. Chipot for his comment on an earlier version of this article.

\def\cprime{$'$}

\end{document}